\documentclass[reqno,onecolumn]{amsart}
\usepackage{amsfonts,amscd,amsthm}
\usepackage{amssymb}
\usepackage{amsmath}
\usepackage{graphicx}
\usepackage{amscd,amsthm}

\newtheorem{theorem}{Theorem}
\theoremstyle{plain}

\newtheorem{definition}{Definition}

\newtheorem{remark}{Remark}

\numberwithin{equation}{section}

\input{tcilatex}

\begin{document}

\begin{center}
\pagestyle{myheadings}\thispagestyle{empty}%
\markboth{\bf  N. Saba and A. Boussayoud}
{\bf Gaussian Mersenne Lucas numbers and polynomials}

\textbf{Gaussian Mersenne Lucas numbers and polynomials}

\textbf{Nabiha Saba}

LMAM Laboratory and Department of Mathematics,

Mohamed Seddik Ben Yahia University, Jijel, Algeria

\textbf{E-Mail:\ sabarnhf1994@gmail.com}

\textbf{Ali Boussayoud}$^{\ast }$

LMAM Laboratory and Department of Mathematics,

Mohamed Seddik Ben Yahia University, Jijel, Algeria

\textbf{E-Mail: aboussayoud@yahoo.fr}

$^{\ast }$\textbf{Corresponding author}

\textbf{\large Abstract}
\end{center}

\begin{quotation}
\qquad In the present article we introduce three new notions which are
called Gaussian Mersenne Lucas numbers, Mersenne Lucas polynomials and
Gaussian Mersenne Lucas polynomials. We present and prove our exciting
properties and results of them such as: recurrence relations, Binet's
formulas, explicit formulas, generating functions, symmetric functions and
negative extensions.
\end{quotation}



\section{\textbf{Introduction and backgrounds}}

In the existing literature, there has been a great interest in the study of
sequences of integers and their applications in various scientific domains.
Some of the sequences that has been extensively studied are the Fibonacci,
Lucas, Mersenne and Mersenne Lucas sequences.\newline
Mersenne Lucas sequence $\left\{ m_{n}\right\} _{n\geq 0}$ is given in \cite%
{NS} by the recurrence relation, 
\begin{equation*}
\left\{ 
\begin{array}{l}
m_{0}=2,\text{ }m_{1}=3 \\ 
m_{n}=3m_{n-1}-2m_{n-2},\text{ for }n\geq 2%
\end{array}%
\right. .
\end{equation*}%
The terms $m_{n}$ of this sequence are known as Mersenne Lucas numbers. Note
that the Mersenne Lucas numbers are given either by the explicit formula
(see \cite{NSaba 2}):%
\begin{equation*}
m_{n}=\tsum\limits_{j=0}^{\left\lfloor \frac{n}{2}\right\rfloor }\left(
-1\right) ^{j}\frac{n}{n-j}\left( 
\begin{array}{c}
n-j \\ 
j%
\end{array}%
\right) 3^{n-2j}2^{j}.
\end{equation*}%
or by the\ Binet's formula \cite{NS}:%
\begin{equation}
m_{n}=2^{n}+1,  \label{1.1}
\end{equation}%
replacing $\left( n\right) $ by $\left( -n\right) $ in the Binet's formula
we get the negative extension of Mersenne Lucas numbers as follows: 
\begin{equation}
m_{-n}=\frac{m_{n}}{2^{n}}.  \label{1.2}
\end{equation}

In literature, there\ have been so many studies of the sequences of Gaussian
numbers, for example Horadam in \cite{Horadam1} examined Fibonacci numbers
on the complex plane and established some interesting properties about them.
Further, Jordan in\ \cite{JORD} studied on Gaussian Fibonacci and Gaussian
Lucas numbers. In addition, Tasci defined and studied Gaussian Mersenne
numbers in \cite{tasc}. Moreover, studies on the different Gaussian
polynomials sequences like Gaussian Jacobsthal, Gaussian Jacobsthal Lucas,
Gaussian Pell and Gaussian Pell Lucas polynomials can be found in the papers 
\cite{ASC}, \cite{HALICI}, \cite{TY}.

\begin{definition}
\cite{Abderrezzak} Let $\lambda $\ and $\mu $ be any two alphabets. We
define $S_{n}(\lambda -\mu )$ by the following form:%
\begin{equation}
\frac{\tprod\limits_{\mu _{i}\in \mu }(1-\mu _{i}z)}{\tprod\limits_{\lambda
_{i}\in \lambda }(1-\lambda _{i}z)}=\sum\limits_{n=0}^{\infty }S_{n}(\lambda
-\mu )z^{n},  \label{1.3}
\end{equation}%
with the condition $S_{n}(\lambda -\mu )=0\ $for\ $n<0.$
\end{definition}

Equation (1.3) can be rewritten in the following form:%
\begin{equation*}
\dsum\limits_{n=0}^{\infty }S_{n}(\lambda -\mu )z^{n}=\left(
\dsum\limits_{n=0}^{\infty }S_{n}(\lambda )z^{n}\right) \times \left(
\dsum\limits_{n=0}^{\infty }S_{n}(-\mu )z^{n}\right) ,
\end{equation*}%
where%
\begin{equation*}
S_{n}(\lambda -\mu )=\dsum\limits_{j=0}^{n}S_{n-j}(-\mu )S_{j}(\lambda ).
\end{equation*}

\begin{remark}
Taking $\lambda =\left\{ 0\right\} $ in (1.3) gives:%
\begin{equation*}
\dsum\limits_{n=0}^{\infty }S_{n}(-\mu )z^{n}=\tprod\limits_{\mu _{i}\in \mu
}(1-\mu _{i}z).
\end{equation*}
\end{remark}

\begin{definition}
\cite{Biya3,nabiha} Let $n$ be positive integer and $\lambda =\left\{
\lambda _{1},\lambda _{2}\right\} \ $are set of given variables. Then, the $%
n^{th}$ symmetric function $S_{n}(\lambda _{1}+\lambda _{2})$ is defined by:%
\begin{equation*}
S_{n}(\lambda )=S_{n}(\lambda _{1}+\lambda _{2})=\frac{\lambda
_{1}^{n+1}-\lambda _{2}^{n+1}}{\lambda _{1}-\lambda _{2}},
\end{equation*}%
with%
\begin{eqnarray*}
S_{0}(\lambda ) &=&S_{0}(\lambda _{1}+\lambda _{2})=1, \\
S_{1}(\lambda ) &=&S_{1}(\lambda _{1}+\lambda _{2})=\lambda _{1}+\lambda
_{2}, \\
S_{2}(\lambda ) &=&S_{2}(\lambda _{1}+\lambda _{2})=\lambda _{1}^{2}+\lambda
_{1}\lambda _{2}+\lambda _{2}^{2}, \\
&&\vdots
\end{eqnarray*}
\end{definition}

This paper is organized as follows: In section 2, we define Gaussian
Mersenne Lucas numbers and we give their interesting properties, and
established relation between classical Mersenne Lucas numbers and Gaussian
Mersenne Lucas numbers. In section 3, we introduce Mersenne Lucas
polynomials and Gaussian Mersenne Lucas polynomials and we obtain some
results related to these polynomials like Binet's formula, explicit formula,
generating function, symmetric function and negative extension.

\section{\textbf{The Gaussian Mersenne Lucas numbers and some properties}}

In this section, we first give the fundamental definition of Gaussian
Mersenne Lucas numbers. Then, we obtain Binet's formula, explicit formula
and negative subscript for these numbers. We first give the following
definition.

\begin{definition}
Let\emph{\ }$n\geq 0$\emph{\ }be integer, the recurrence relation of
Gaussian Mersenne Lucas numbers $\left\{ Gm_{n}\right\} _{n\geq 0}$ is given
as:%
\begin{equation}
Gm_{n}:=\left\{ 
\begin{array}{c}
2+\frac{3i}{2},\text{ \ \ \ \ \ \ \ \ \ \ \ \ \ \ \ \ \ \ \ \ \ \ \ \ \ \ \
\ \ \ \ \ \ \ \ \ \ \ \ \ \ \textit{if} }n=0 \\ 
3+2i,\text{ \ \ \ \ \ \ \ \ \ \ \ \ \ \ \ \ \ \ \ \ \ \ \ \ \ \ \ \ \ \ \ \
\ \ \ \ \ \ \ \ \ if }n=1 \\ 
3Gm_{n-1}-2Gm_{n-2},\text{\ \ \ \ \ \ \ \ \ \ \ \ \ \ \ \ \ \ \ \ \ \ \ \ if 
}n\geq 2%
\end{array}%
\right. ,  \label{2.1}
\end{equation}
\end{definition}

It can be easily seen that:%
\begin{equation}
Gm_{n}=m_{n}+im_{n-1},\text{ for }n\geq 1,  \label{2.2}
\end{equation}%
with $m_{n}$ is the $n^{th}$ Mersenne Lucas numbers.

For later use the first few terms of Gaussian Mersenne Lucas numbers are as
shown in the following table:%
\begin{equation*}
\begin{tabular}{|c|c|c|c|c|c|c|c|}
\hline
$n$ & $0$ & $1$ & $2$ & $3$ & $4$ & $5$ & $...$ \\ \hline
$Gm_{n}$ & $2+\frac{3i}{2}$ & $3+2i$ & $5+3i$ & $9+5i$ & $17+9i$ & $33+17i$
& $...$ \\ \hline
\end{tabular}%
\end{equation*}

\begin{center}
\textbf{Table 1. }Gaussian Mersenne Lucas numbers for $0\leq n\leq 5.$
\end{center}

Let $\lambda _{1}$ and $\lambda _{2}$ be the roots of the characteristic
equation $\lambda ^{2}-3\lambda +2=0$ of the recurrence relation (2.1). Then%
\begin{equation*}
\lambda _{1}=2\text{ and }\lambda _{2}=1.
\end{equation*}

Now we can give the $n^{th}$ term of Gaussian Mersenne Lucas numbers.

\begin{theorem}
For $n\geq 0,$ the Binet's formula for Gaussian Mersenne Lucas numbers is
given by:%
\begin{equation}
Gm_{n}=2^{n}+1+i\left( 2^{n-1}+1\right) .  \label{2.3}
\end{equation}
\end{theorem}

\begin{proof}
We know that the general solution for the recurrence relation of Gaussian
Mersenne Lucas numbers given by:%
\begin{equation*}
Gm_{n}=c2^{n}+d,
\end{equation*}%
where $c$ and $d$ are the coefficients.

Using the initial values $Gm_{0}=2+\frac{3i}{2}$ and $Gm_{1}=3+2i,$ we
obtain:%
\begin{equation*}
c+d=2+\frac{3i}{2}\text{ and }2c+d=1.
\end{equation*}%
By these equalities we get:%
\begin{equation*}
c=1+\frac{i}{2}\text{ and }d=1+i.
\end{equation*}%
Therefore:%
\begin{eqnarray*}
Gm_{n} &=&\left( 1+\frac{i}{2}\right) 2^{n}+1+i \\
&=&2^{n}+1+i\left( 2^{n-1}+1\right) .
\end{eqnarray*}%
As required.
\end{proof}

The negative extension of Gaussian Mersenne Lucas numbers $\left(
Gm_{-n}\right) $\ gives in the next theorem.

\begin{theorem}
Let $n$ be any positive integer. Then we have:%
\begin{equation}
Gm_{-n}=\frac{1}{2^{n}}\left[ m_{n}+\frac{i}{2}m_{n+1}\right] ,  \label{2.4}
\end{equation}%
with $m_{n}$ is the $n^{th}$Mersenne Lucas numbers.
\end{theorem}

\begin{proof}
By \cite{NS}, we have the negative extension of Mersenne Lucas numbers is
given by:%
\begin{equation*}
m_{-n}=\frac{m_{n}}{2^{n}}.
\end{equation*}

By using the Eq. (2.2), we easily obtain:%
\begin{eqnarray*}
Gm_{-n} &=&m_{-n}+im_{-n-1} \\
&=&\frac{m_{n}}{2^{n}}+i\frac{m_{n+1}}{2^{n+1}} \\
&=&\frac{1}{2^{n}}\left[ m_{n}+\frac{i}{2}m_{n+1}\right] .
\end{eqnarray*}

Hence, we obtain the desired result.
\end{proof}

Now, we aim to give the explicit formula for Gaussian Mersenne Lucas
numbers. For this purpose, we shall prove the following theorem.

\begin{theorem}
The explicit formula for Gaussian Mersenne Lucas numbers is given by:%
\begin{equation}
Gm_{n}=\tsum\limits_{j=0}^{\left\lfloor \frac{n}{2}\right\rfloor }\left(
-1\right) ^{j}\frac{n}{n-j}\left( 
\begin{array}{c}
n-j \\ 
j%
\end{array}%
\right) 3^{n-2j}2^{j}+i\tsum\limits_{j=0}^{\left\lfloor \frac{n-1}{2}%
\right\rfloor }\left( -1\right) ^{j}\frac{n-1}{n-j-1}\left( 
\begin{array}{c}
n-j-1 \\ 
j%
\end{array}%
\right) 3^{n-2j-1}2^{j}.  \label{2.5}
\end{equation}
\end{theorem}

\begin{proof}
By \cite{NSaba 2}, we have the explicit formula of the $n^{th}$ Mersenne
Lucas numbers is given by:%
\begin{equation*}
m_{n}=\tsum\limits_{j=0}^{\left\lfloor \frac{n}{2}\right\rfloor }\left(
-1\right) ^{j}\frac{n}{n-j}\left( 
\begin{array}{c}
n-j \\ 
j%
\end{array}%
\right) 3^{n-2j}2^{j}.
\end{equation*}%
By using the Eq. (2.2), we easily obtain: 
\begin{eqnarray*}
Gm_{n} &=&m_{n}+im_{n-1} \\
&=&\tsum\limits_{j=0}^{\left\lfloor \frac{n}{2}\right\rfloor }\left(
-1\right) ^{j}\frac{n}{n-j}\left( 
\begin{array}{c}
n-j \\ 
j%
\end{array}%
\right) 3^{n-2j}2^{j}+i\tsum\limits_{j=0}^{\left\lfloor \frac{n-1}{2}%
\right\rfloor }\left( -1\right) ^{j}\frac{n-1}{n-j-1}\left( 
\begin{array}{c}
n-j-1 \\ 
j%
\end{array}%
\right) 3^{n-2j-1}2^{j}.
\end{eqnarray*}%
As required.
\end{proof}

\begin{theorem}
For $n\in 
%TCIMACRO{\U{2115} }%
%BeginExpansion
\mathbb{N}
%EndExpansion
,$ the generating function of Gaussian Mersenne Lucas numbers is given by:%
\begin{equation}
\sum_{n=0}^{\infty }Gm_{n}z^{n}=\frac{4+3i-\left( 6+5i\right) z}{2-6z+4z^{2}}%
.  \label{2.6}
\end{equation}
\end{theorem}

\begin{proof}
The Gaussian Mersenne Lucas numbers can be considered as the coefficients of
the formal power series:%
\begin{equation*}
g(z)=\sum_{n=0}^{\infty }Gm_{n}z^{n}.
\end{equation*}%
Using the initial conditions, we get:%
\begin{eqnarray*}
g(z) &=&Gm_{0}+Gm_{1}z+\sum_{n=2}^{\infty }Gm_{n}z^{n} \\
&=&Gm_{0}+Gm_{1}z+\sum_{n=2}^{\infty }\left( 3Gm_{n-1}-2Gm_{n-2}\right) z^{n}
\\
&=&Gm_{0}+Gm_{1}z+3z\sum_{n=1}^{\infty }Gm_{n}z^{n}-2z^{2}\sum_{n=0}^{\infty
}Gm_{n}z^{n} \\
&=&Gm_{0}+\left( Gm_{1}-3Gm_{0}\right) z+3z\sum_{n=0}^{\infty
}Gm_{n}z^{n}-2z^{2}\sum_{n=0}^{\infty }Gm_{n}z^{n} \\
&=&2+\frac{3i}{2}+\left( \frac{5i}{2}-3\right) z+\left( 3z-2z^{2}\right)
g(z).
\end{eqnarray*}%
Hence, we obtain:%
\begin{equation*}
\left( 1-3z+2z^{2}\right) g(z)=2+\frac{3i}{2}-\left( \frac{5i}{2}+3\right) z.
\end{equation*}%
Therefore:%
\begin{equation*}
g(z)=\frac{4+3i-\left( 6+5i\right) z}{2-6z+4z^{2}}.
\end{equation*}%
Hence, we obtain the desired result.
\end{proof}

\begin{theorem}
For $n\in 
%TCIMACRO{\U{2115} }%
%BeginExpansion
\mathbb{N}
%EndExpansion
,$ the symmetric function\textbf{\ }of Gaussian Mersenne Lucas numbers is
given by:%
\begin{equation}
Gm_{n}=\left( 2+\frac{3i}{2}\right) S_{n}\left( \lambda _{1}+\left[ -\lambda
_{2}\right] \right) -\left( \frac{5i}{2}+3\right) S_{n-1}\left( \lambda _{1}+%
\left[ -\lambda _{2}\right] \right) .  \label{2.7}
\end{equation}%
with $\lambda _{1}=2$ and $\lambda _{2}=1.$
\end{theorem}

\begin{proof}
By \cite{Bouss5}, we have:%
\begin{equation}
\dsum\limits_{n=0}^{\infty }S_{n}\left( \lambda _{1}+\left[ -\lambda _{2}%
\right] \right) z^{n}=\frac{1}{1-\left( \lambda _{1}-\lambda _{2}\right)
z-\lambda _{1}\lambda _{2}z^{2}},\text{ with }\lambda =\left\{ \lambda
_{1},\lambda _{2}\right\} \text{\ an alphabet.}  \label{2.8}
\end{equation}%
From this equation we get:%
\begin{equation}
\dsum\limits_{n=0}^{\infty }S_{n-1}\left( \lambda _{1}+\left[ -\lambda _{2}%
\right] \right) z^{n}=\frac{z}{1-\left( \lambda _{1}-\lambda _{2}\right)
z-\lambda _{1}\lambda _{2}z^{2}}.  \label{2.9}
\end{equation}%
Setting $\lambda _{1}-\lambda _{2}=3$ and $\lambda _{1}\lambda _{2}=-2$ in
the Eqs. (2.8) and (2.9), we get:%
\begin{eqnarray}
\sum_{n=0}^{\infty }S_{n}\left( \lambda _{1}+\left[ -\lambda _{2}\right]
\right) z^{n} &=&\frac{1}{1-3z+2z^{2}},  \label{2.10} \\
\sum_{n=0}^{\infty }S_{n-1}\left( \lambda _{1}+\left[ -\lambda _{2}\right]
\right) z^{n} &=&\frac{z}{1-3z+2z^{2}}.  \label{2.11}
\end{eqnarray}

Multiplying the equation (2.10) by $\left( 2+\frac{3i}{2}\right) $ and
adding it to the equation obtained by (2.11) multiplying by $\left( -\frac{5i%
}{2}-3\right) $, then we obtain the following equality:%
\begin{eqnarray*}
\sum_{n=0}^{\infty }\left( \left( 2+\frac{3i}{2}\right) S_{n}\left( \lambda
_{1}+\left[ -\lambda _{2}\right] \right) -\left( \frac{5i}{2}+3\right)
S_{n-1}\left( \lambda _{1}+\left[ -\lambda _{2}\right] \right) \right) z^{n}
&=&\frac{4+3i-\left( 6+5i\right) z}{2-6z+4z^{2}} \\
&=&\sum_{n=0}^{\infty }Gm_{n}z^{n}.
\end{eqnarray*}%
Comparing of the coefficients of $z^{n}$ we obtain:%
\begin{equation*}
Gm_{n}=\left( 2+\frac{3i}{2}\right) S_{n}\left( \lambda _{1}+\left[ -\lambda
_{2}\right] \right) -\left( \frac{5i}{2}+3\right) S_{n-1}\left( \lambda _{1}+%
\left[ -\lambda _{2}\right] \right) .
\end{equation*}%
This completes the proof.
\end{proof}

Now, we aim to give the generating functions for odd and even Gaussian
Mersenne Lucas numbers.

\begin{theorem}
For $n\in 
%TCIMACRO{\U{2115} }%
%BeginExpansion
\mathbb{N}
%EndExpansion
,$ the generating functions of even and odd Gaussian Mersenne Lucas numbers
are respectively given by:%
\begin{eqnarray}
\sum\limits_{n=0}^{\infty }Gm_{2n}z^{n} &=&\frac{4+3i-\left( 10+9i\right) z}{%
2-10z+8z^{2}},  \label{2.13} \\
\sum\limits_{n=0}^{\infty }Gm_{2n+1}z^{n} &=&\frac{6+4i-\left( 12+10i\right)
z}{2-10z+8z^{2}}.  \label{2.14}
\end{eqnarray}
\end{theorem}

\begin{proof}
By \cite{nana}, we have:%
\begin{eqnarray}
\sum\limits_{n=0}^{\infty }S_{2n-1}\left( \lambda _{1}+\left[ -\lambda _{2}%
\right] \right) z^{n} &=&\frac{\left( \lambda _{1}-\lambda _{2}\right) z}{%
1-\left( \left( \lambda _{1}-\lambda _{2}\right) ^{2}+2\lambda _{1}\lambda
_{2}\right) z+\lambda _{1}^{2}\lambda _{2}^{2}z^{2}},  \label{2.15} \\
\sum\limits_{n=0}^{\infty }S_{2n}\left( \lambda _{1}+\left[ -\lambda _{2}%
\right] \right) z^{n} &=&\frac{1-\lambda _{1}\lambda _{2}z}{1-\left( \left(
\lambda _{1}-\lambda _{2}\right) ^{2}+2\lambda _{1}\lambda _{2}\right)
z+\lambda _{1}^{2}\lambda _{2}^{2}z^{2}},  \label{2.16} \\
\sum\limits_{n=0}^{\infty }S_{2n+1}\left( \lambda _{1}+\left[ -\lambda _{2}%
\right] \right) z^{n} &=&\frac{\lambda _{1}-\lambda _{2}}{1-\left( \left(
\lambda _{1}-\lambda _{2}\right) ^{2}+2\lambda _{1}\lambda _{2}\right)
z+\lambda _{1}^{2}\lambda _{2}^{2}z^{2}},  \label{2.17}
\end{eqnarray}%
with $\lambda =\left\{ \lambda _{1},\lambda _{2}\right\} $\ an alphabet.
Setting $\lambda _{1}-\lambda _{2}=3$ and $\lambda _{1}\lambda _{2}=-2$ in
the Eqs. (2.15), (2.16) and (2.17), we get:%
\begin{eqnarray}
\sum\limits_{n=0}^{\infty }S_{2n-1}\left( \lambda _{1}+\left[ -\lambda _{2}%
\right] \right) z^{n} &=&\frac{3z}{1-5z+4z^{2}},  \label{2.18} \\
\sum\limits_{n=0}^{\infty }S_{2n}\left( \lambda _{1}+\left[ -\lambda _{2}%
\right] \right) z^{n} &=&\frac{1+2z}{1-5z+4z^{2}},  \label{2.19} \\
\sum\limits_{n=0}^{\infty }S_{2n+1}\left( \lambda _{1}+\left[ -\lambda _{2}%
\right] \right) z^{n} &=&\frac{3}{1-5z+4z^{2}}.  \label{2.20}
\end{eqnarray}%
Writing $\left( 2n\right) $ instead of $\left( n\right) $ in the Eq. (2.7),
we get:%
\begin{eqnarray*}
\sum\limits_{n=0}^{\infty }Gm_{2n}z^{n} &=&\sum\limits_{n=0}^{\infty }\left(
\left( 2+\frac{3i}{2}\right) S_{2n}\left( \lambda _{1}+\left[ -\lambda _{2}%
\right] \right) -\left( \frac{5i}{2}+3\right) S_{2n-1}\left( \lambda _{1}+%
\left[ -\lambda _{2}\right] \right) \right) z^{n} \\
&=&\left( 2+\frac{3i}{2}\right) \sum\limits_{n=0}^{\infty }S_{2n}\left(
\lambda _{1}+\left[ -\lambda _{2}\right] \right) z^{n}-\left( \frac{5i}{2}%
+3\right) \sum\limits_{n=0}^{\infty }S_{2n-1}\left( \lambda _{1}+\left[
-\lambda _{2}\right] \right) z^{n} \\
&=&\frac{\left( 4+3i\right) \left( 1+2z\right) }{2\left( 1-5z+4z^{2}\right) }%
-\frac{\left( 6+5i\right) 3z}{2\left( 1-5z+4z^{2}\right) } \\
&=&\frac{4+3i-\left( 10+9i\right) z}{2-10z+8z^{2}}.
\end{eqnarray*}

Which is the generating function of even Gaussian Mersenne Lucas numbers.%
\newline
Substituting\ $n$\ by $\left( 2n+1\right) $ in the Eq. (2.7), we obtain:%
\begin{eqnarray*}
\sum\limits_{n=0}^{\infty }Gm_{2n+1}z^{n} &=&\sum\limits_{n=0}^{\infty
}\left( \left( 2+\frac{3i}{2}\right) S_{2n+1}\left( \lambda _{1}+\left[
-\lambda _{2}\right] \right) -\left( \frac{5i}{2}+3\right) S_{2n}\left(
\lambda _{1}+\left[ -\lambda _{2}\right] \right) \right) z^{n} \\
&=&\left( 2+\frac{3i}{2}\right) \sum\limits_{n=0}^{\infty }S_{2n+1}\left(
\lambda _{1}+\left[ -\lambda _{2}\right] \right) z^{n}-\left( \frac{5i}{2}%
+3\right) \sum\limits_{n=0}^{\infty }S_{2n}\left( \lambda _{1}+\left[
-\lambda _{2}\right] \right) z^{n} \\
&=&\frac{3\left( 4+3i\right) }{2\left( 1-5z+4z^{2}\right) }-\frac{\left(
6+5i\right) \left( 1+2z\right) }{2\left( 1-5z+4z^{2}\right) } \\
&=&\frac{6+4i-\left( 12+10i\right) z}{2-10z+8z^{2}}.
\end{eqnarray*}%
Which is the generating function of odd Gaussian Mersenne Lucas numbers.
\end{proof}

\section{\textbf{Mersenne Lucas polynomials and Gaussian Mersenne Lucas
polynomials and their some interesting properties}}

In this section, we give the definitions of Mersenne Lucas and Gaussian
Mersenne Lucas polynomials, and we obtain some interesting properties.

\begin{definition}
Let\emph{\ }$n\geq 0$\emph{\ }be integer, the recurrence relation of
Mersenne Lucas polynomials $\left\{ m_{n}\left( x\right) \right\} _{n\geq 0}$
is given as:%
\begin{equation}
m_{n}\left( x\right) :=\left\{ 
\begin{array}{c}
2,\text{ \ \ \ \ \ \ \ \ \ \ \ \ \ \ \ \ \ \ \ \ \ \ \ \ \ \ \ \ \ \ \ \ \ \
\ \ \ \ \ \ \ \ \textit{if} }n=0 \\ 
3x,\text{ \ \ \ \ \ \ \ \ \ \ \ \ \ \ \ \ \ \ \ \ \ \ \ \ \ \ \ \ \ \ \ \ \
\ \ \ \ \ \ \ if }n=1 \\ 
3xm_{n-1}\left( x\right) -2m_{n-2}\left( x\right) ,\text{\ \ \ \ \ \ \ \ \ \
\ \ if }n\geq 2%
\end{array}%
\right. ,  \label{3.1}
\end{equation}
\end{definition}

\begin{definition}
Let\emph{\ }$n\geq 0$\emph{\ }be integer, the recurrence relation of
Gaussian Mersenne Lucas polynomials $\left\{ Gm_{n}\left( x\right) \right\}
_{n\geq 0}$ is given as:%
\begin{equation}
Gm_{n}\left( x\right) :=\left\{ 
\begin{array}{c}
2+\frac{3i}{2}x,\text{ \ \ \ \ \ \ \ \ \ \ \ \ \ \ \ \ \ \ \ \ \ \ \ \ \ \ \
\ \ \ \ \ \ \ \textit{if} }n=0 \\ 
3x+2i,\text{ \ \ \ \ \ \ \ \ \ \ \ \ \ \ \ \ \ \ \ \ \ \ \ \ \ \ \ \ \ \ \ \
\ \ \ if }n=1 \\ 
3xGm_{n-1}\left( x\right) -2Gm_{n-2}\left( x\right) ,\text{\ \ \ \ \ \ \ \
if }n\geq 2%
\end{array}%
\right. ,  \label{3.2}
\end{equation}
\end{definition}

It is easily seen that:%
\begin{equation}
Gm_{n}\left( x\right) =m_{n}\left( x\right) +im_{n-1}\left( x\right) ,\text{
for }n\geq 1.  \label{3.3}
\end{equation}

The first few terms of Mersenne Lucas and Gaussian Mersenne Lucas
polynomials are as shown in the following table:%
\begin{equation*}
\begin{tabular}{|c|c|c|}
\hline
$n$ & $m_{n}\left( x\right) $ & $Gm_{n}\left( x\right) $ \\ \hline
$0$ & $2$ & $2+\frac{3i}{2}x$ \\ \hline
$1$ & $3x$ & $3x+2i$ \\ \hline
$2$ & $9x^{2}-4$ & $9x^{2}-4+3ix$ \\ \hline
$3$ & $27x^{3}-18x$ & $27x^{3}-18x+i\left( 9x^{2}-4\right) $ \\ \hline
$4$ & $81x^{4}-72x^{2}+8$ & $81x^{4}-72x^{2}+8+i\left( 27x^{3}-18x\right) $
\\ \hline
$5$ & $243x^{5}-270x^{3}+60x$ & $243x^{5}-270x^{3}+60x+i\left(
81x^{4}-72x^{2}+8\right) $ \\ \hline
$\vdots $ & $\vdots $ & $\vdots $ \\ \hline
\end{tabular}%
\end{equation*}

\begin{center}
\textbf{Table 2. }Mersenne Lucas and Gaussian Mersenne Lucas polynomials for 
$0\leq n\leq 5.$
\end{center}

Let $\lambda _{1}$ and $\lambda _{2}$ be the roots of the characteristic
equation $\lambda ^{2}-3x\lambda +2=0$ of the recurrence relations (3.1) and
(3.2). Then%
\begin{equation*}
\lambda _{1}=\frac{3x+\sqrt{9x^{2}-8}}{2}\text{ and }\lambda _{2}=\frac{3x-%
\sqrt{9x^{2}-8}}{2}.
\end{equation*}%
Note\ that:%
\begin{equation*}
\lambda _{1}+\lambda _{2}=3x,\text{ }\lambda _{1}\lambda _{2}=2\text{ and }%
\lambda _{1}-\lambda _{2}=\sqrt{9x^{2}-8}.
\end{equation*}

Now we can give the $n^{th}$ term of Mersenne Lucas and Gaussian Mersenne
Lucas polynomials.

\begin{theorem}
For $n\geq 0,$ the Binet's formulas for Mersenne Lucas and Gaussian Mersenne
Lucas polynomials are respectively given by:%
\begin{eqnarray}
m_{n}\left( x\right) &=&\lambda _{1}^{n}+\lambda _{2}^{n},  \label{3.4} \\
Gm_{n}\left( x\right) &=&\lambda _{1}^{n}+\lambda _{2}^{n}+i\left( \lambda
_{1}^{n-1}+\lambda _{2}^{n-1}\right) .  \label{3.5}
\end{eqnarray}
\end{theorem}

\begin{proof}
We know that the general solution for the recurrence relation of Mersenne
Lucas polynomials given by:%
\begin{equation*}
m_{n}\left( x\right) =c\lambda _{1}^{n}+d\lambda _{2}^{n},
\end{equation*}%
where $c$ and $d$ are the coefficients.

Using the initial values $m_{0}\left( x\right) =2$ and $m_{1}\left( x\right)
=3x,$ we obtain:%
\begin{equation*}
c+d=2\text{ and }c\lambda _{1}+d\lambda _{2}=3x.
\end{equation*}%
By these equalities we get:%
\begin{equation*}
c=d=1\text{ }.
\end{equation*}%
Therefore:%
\begin{equation*}
m_{n}\left( x\right) =\lambda _{1}^{n}+\lambda _{2}^{n}.
\end{equation*}%
Which\ is the first equation. Second equation can be proved similarly\textit{%
.}
\end{proof}

The negative extensions of Mersenne Lucas and Gaussian Mersenne Lucas
polynomials $\left( m_{-n}\left( x\right) \right) $\ and $\left(
Gm_{-n}\left( x\right) \right) $\ gives in the next theorem.

\begin{theorem}
Let $n$ be any positive integer. Then we have:%
\begin{eqnarray}
m_{-n}\left( x\right) &=&\frac{m_{n}\left( x\right) }{2^{n}},  \label{3.6} \\
Gm_{-n}\left( x\right) &=&\frac{1}{2^{n}}\left[ m_{n}\left( x\right) +\frac{i%
}{2}m_{n+1}\left( x\right) \right] .  \label{3.7}
\end{eqnarray}
\end{theorem}

\begin{proof}
Replacing $\left( n\right) $\ by $(-n)$\ in the Binet's formula (3.4), we
can write:%
\begin{eqnarray*}
m_{-n}\left( x\right) &=&\lambda _{1}^{-n}+\lambda _{2}^{-n} \\
&=&\frac{1}{\lambda _{1}^{n}}+\frac{1}{\lambda _{2}^{n}} \\
&=&\frac{\lambda _{1}^{n}+\lambda _{2}^{n}}{2^{n}} \\
&=&\frac{m_{n}\left( x\right) }{2^{n}}.
\end{eqnarray*}%
Which\ is the negative extension of Mersenne Lucas polynomials.\newline
According the Eq. (3.3), we get:%
\begin{eqnarray*}
Gm_{-n}\left( x\right) &=&m_{-n}\left( x\right) +im_{-n-1}\left( x\right) \\
&=&\frac{m_{n}\left( x\right) }{2^{n}}+i\frac{m_{n+1}\left( x\right) }{%
2^{n+1}} \\
&=&\frac{1}{2^{n}}\left[ m_{n}\left( x\right) +\frac{i}{2}m_{n+1}\left(
x\right) \right] .
\end{eqnarray*}%
Which is the negative extension of Gaussian Mersenne Lucas polynomials.
\end{proof}

\begin{theorem}
For $n\in 
%TCIMACRO{\U{2115} }%
%BeginExpansion
\mathbb{N}
%EndExpansion
,$ the generating functions of Mersenne Lucas and Gaussian Mersenne Lucas
polynomials are respectively given by:%
\begin{eqnarray}
\sum_{n=0}^{\infty }m_{n}\left( x\right) z^{n} &=&\frac{2-3xz}{1-3xz+2z^{2}},
\label{3.8} \\
\sum_{n=0}^{\infty }Gm_{n}\left( x\right) z^{n} &=&\frac{4+3ix+\left(
i\left( 4-9x^{2}\right) -6x\right) z}{2-6xz+4z^{2}}.  \label{3.9}
\end{eqnarray}
\end{theorem}

\begin{proof}
The Mersenne Lucas polynomials can be considered as the coefficients of the
formal power series:%
\begin{equation*}
g(z)=\sum_{n=0}^{\infty }m_{n}\left( x\right) z^{n}.
\end{equation*}%
Using the initial conditions, we get:%
\begin{eqnarray*}
g(z) &=&m_{0}\left( x\right) +m_{1}\left( x\right) z+\sum_{n=2}^{\infty
}m_{n}\left( x\right) z^{n} \\
&=&m_{0}\left( x\right) +m_{1}\left( x\right) z+\sum_{n=2}^{\infty }\left(
3xm_{n-1}\left( x\right) -2m_{n-2}\left( x\right) \right) z^{n} \\
&=&m_{0}\left( x\right) +m_{1}\left( x\right) z+3xz\sum_{n=1}^{\infty
}m_{n}\left( x\right) z^{n}-2z^{2}\sum_{n=0}^{\infty }m_{n}\left( x\right)
z^{n} \\
&=&m_{0}\left( x\right) +\left( m_{1}\left( x\right) -3xm_{0}\left( x\right)
\right) z+3xz\sum_{n=0}^{\infty }m_{n}\left( x\right)
z^{n}-2z^{2}\sum_{n=0}^{\infty }m_{n}\left( x\right) z^{n} \\
&=&2-3xz+\left( 3xz-2z^{2}\right) g(z).
\end{eqnarray*}%
Hence, we obtain:%
\begin{equation*}
\left( 1-3xz+2z^{2}\right) g(z)=2-3xz.
\end{equation*}%
Therefore:%
\begin{equation*}
g(z)=\frac{2-3xz}{1-3xz+2z^{2}}.
\end{equation*}

Which\textit{\ }gives\textit{\ }equation (3.8). Using the same procedure, we
can obtain equation (3.9).
\end{proof}

\begin{theorem}
For $n\in 
%TCIMACRO{\U{2115} }%
%BeginExpansion
\mathbb{N}
%EndExpansion
,$ the symmetric functions\textbf{\ }of Mersenne Lucas and Gaussian Mersenne
Lucas polynomials are respectively given by:%
\begin{eqnarray}
m_{n}\left( x\right) &=&2S_{n}\left( \lambda _{1}+\left[ -\lambda _{2}\right]
\right) -3xS_{n-1}\left( \lambda _{1}+\left[ -\lambda _{2}\right] \right) ,
\label{3.10} \\
Gm_{n}\left( x\right) &=&\left( 2+\frac{3ix}{2}\right) S_{n}\left( \lambda
_{1}+\left[ -\lambda _{2}\right] \right) +\left( i\left( 2-\frac{9}{2}%
x^{2}\right) -3x\right) S_{n-1}\left( \lambda _{1}+\left[ -\lambda _{2}%
\right] \right) .  \label{3.11}
\end{eqnarray}%
with $\lambda _{1}=\frac{3x+\sqrt{9x^{2}-8}}{2}$ and $\lambda _{2}=\frac{3x-%
\sqrt{9x^{2}-8}}{2}.$
\end{theorem}

\begin{proof}
Setting $\lambda _{1}-\lambda _{2}=3x$ and $\lambda _{1}\lambda _{2}=-2$ in
the Eqs. (2.8) and (2.9), we get:%
\begin{eqnarray}
\sum_{n=0}^{\infty }S_{n}\left( \lambda _{1}+\left[ -\lambda _{2}\right]
\right) z^{n} &=&\frac{1}{1-3xz+2z^{2}},  \label{3.12} \\
\sum_{n=0}^{\infty }S_{n-1}\left( \lambda _{1}+\left[ -\lambda _{2}\right]
\right) z^{n} &=&\frac{z}{1-3xz+2z^{2}}.  \label{3.13}
\end{eqnarray}

Multiplying the equation (3.12) by $\left( 2\right) $ and adding it to the
equation obtained by (3.13) multiplying by $\left( -3x\right) $, then we
obtain the following equality:%
\begin{equation*}
\sum_{n=0}^{\infty }\left( 2S_{n}\left( \lambda _{1}+\left[ -\lambda _{2}%
\right] \right) -3xS_{n-1}\left( \lambda _{1}+\left[ -\lambda _{2}\right]
\right) \right) z^{n}=\frac{2-3xz}{1-3xz+2z^{2}}=\sum_{n=0}^{\infty
}m_{n}\left( x\right) z^{n}.
\end{equation*}%
Comparing of the coefficients of $z^{n}$ we obtain:%
\begin{equation*}
m_{n}\left( x\right) =2S_{n}\left( \lambda _{1}+\left[ -\lambda _{2}\right]
\right) -3xS_{n-1}\left( \lambda _{1}+\left[ -\lambda _{2}\right] \right) .
\end{equation*}%
And multiplying the equation (3.12) by $\left( 2+\frac{3ix}{2}\right) $ and
adding it to the equation obtained by (3.13) multiplying by $\left( i\left(
2-\frac{9}{2}x^{2}\right) -3x\right) $, then we obtain the following
equality:%
\begin{eqnarray*}
\sum_{n=0}^{\infty }\left( \left( 2+\frac{3ix}{2}\right) S_{n}\left( \lambda
_{1}+\left[ -\lambda _{2}\right] \right) +\left( i\left( 2-\frac{9}{2}%
x^{2}\right) -3x\right) S_{n-1}\left( \lambda _{1}+\left[ -\lambda _{2}%
\right] \right) \right) z^{n} &=&\frac{4+3ix+\left( i\left( 4-9x^{2}\right)
-6x\right) z}{2-6xz+4z^{2}} \\
&=&\sum_{n=0}^{\infty }Gm_{n}\left( x\right) z^{n}.
\end{eqnarray*}%
Comparing of the coefficients of $z^{n}$ we obtain:%
\begin{equation*}
Gm_{n}\left( x\right) =\left( 2+\frac{3ix}{2}\right) S_{n}\left( \lambda
_{1}+\left[ -\lambda _{2}\right] \right) +\left( i\left( 2-\frac{9}{2}%
x^{2}\right) -3x\right) S_{n-1}\left( \lambda _{1}+\left[ -\lambda _{2}%
\right] \right) .
\end{equation*}

This completes the proof.
\end{proof}

Now, we aim to give the explicit formulas for Mersenne Lucas and Gaussian
Mersenne Lucas polynomials. For this purpose, we shall prove the following
theorem.

\begin{theorem}
The explicit formulas for Mersenne Lucas and Gaussian Mersenne Lucas
polynomials are respectively given by:%
\begin{equation}
m_{n}\left( x\right) =\tsum\limits_{j=0}^{\left\lfloor \frac{n}{2}%
\right\rfloor }\left( -1\right) ^{j}\frac{n}{n-j}\left( 
\begin{array}{c}
n-j \\ 
j%
\end{array}%
\right) \left( 3x\right) ^{n-2j}2^{j},  \label{3.14}
\end{equation}%
\begin{equation}
Gm_{n}\left( x\right) =\tsum\limits_{j=0}^{\left\lfloor \frac{n}{2}%
\right\rfloor }\left( -1\right) ^{j}\frac{n}{n-j}\left( 
\begin{array}{c}
n-j \\ 
j%
\end{array}%
\right) \left( 3x\right) ^{n-2j}2^{j}+i\tsum\limits_{j=0}^{\left\lfloor 
\frac{n-1}{2}\right\rfloor }\left( -1\right) ^{j}\frac{n-1}{n-j-1}\left( 
\begin{array}{c}
n-j-1 \\ 
j%
\end{array}%
\right) \left( 3x\right) ^{n-2j-1}2^{j}.  \label{3.15}
\end{equation}
\end{theorem}

\begin{proof}
By \cite{NN}, we have:%
\begin{equation}
S_{n}\left( \lambda _{1}+\left[ -\lambda _{2}\right] \right)
=\tsum\limits_{j=0}^{\left\lfloor \frac{n}{2}\right\rfloor }\left( 
\begin{array}{c}
n-j \\ 
j%
\end{array}%
\right) \left( \lambda _{1}-\lambda _{2}\right) ^{n-2j}\left( \lambda
_{1}\lambda _{2}\right) ^{j}.  \label{3.16}
\end{equation}%
\begin{equation}
S_{n-1}\left( \lambda _{1}+\left[ -\lambda _{2}\right] \right)
=\tsum\limits_{j=0}^{\left\lfloor \frac{n-1}{2}\right\rfloor }\left( 
\begin{array}{c}
n-j-1 \\ 
j%
\end{array}%
\right) \left( \lambda _{1}-\lambda _{2}\right) ^{n-2j-1}\left( \lambda
_{1}\lambda _{2}\right) ^{j}.  \label{3.17}
\end{equation}%
Setting $\left\{ 
\begin{array}{c}
\lambda _{1}-\lambda _{2}=3x \\ 
\lambda _{1}\lambda _{2}=-2%
\end{array}%
\right. $ in the Eqs. (3.16) and (3.17), we get:%
\begin{equation}
S_{n}\left( \lambda _{1}+\left[ -\lambda _{2}\right] \right)
=\dsum\limits_{j=0}^{\left\lfloor \frac{n}{2}\right\rfloor }\left( -1\right)
^{j}\left( 
\begin{array}{c}
n-j \\ 
j%
\end{array}%
\right) \left( 3x\right) ^{n-2j}2^{j},  \label{3.18}
\end{equation}%
\begin{equation}
S_{n-1}\left( \lambda _{1}+\left[ -\lambda _{2}\right] \right)
=\dsum\limits_{j=0}^{\left\lfloor \frac{n-1}{2}\right\rfloor }\left(
-1\right) ^{j}\left( 
\begin{array}{c}
n-j-1 \\ 
j%
\end{array}%
\right) \left( 3x\right) ^{n-2j-1}2^{j}.  \label{3.19}
\end{equation}

On the other hand, by Eq. (3.10) we have:%
\begin{equation*}
m_{n}\left( x\right) =2S_{n}\left( \lambda _{1}+\left[ -\lambda _{2}\right]
\right) -3xS_{n-1}\left( \lambda _{1}+\left[ -\lambda _{2}\right] \right) .
\end{equation*}

Multiplying the equation (3.18) by $\left( 2\right) $ and adding it to the
equation obtained by (3.19) multiplying by $\left( -3x\right) $, then we get:%
\begin{eqnarray*}
m_{n}\left( x\right) &=&2\dsum\limits_{j=0}^{\left\lfloor \frac{n}{2}%
\right\rfloor }\left( -1\right) ^{j}\left( 
\begin{array}{c}
n-j \\ 
j%
\end{array}%
\right) \left( 3x\right) ^{n-2j}2^{j}-3x\dsum\limits_{j=0}^{\left\lfloor 
\frac{n-1}{2}\right\rfloor }\left( -1\right) ^{j}\left( 
\begin{array}{c}
n-j-1 \\ 
j%
\end{array}%
\right) \left( 3x\right) ^{n-2j-1}2^{j} \\
&=&2\dsum\limits_{j=0}^{\left\lfloor \frac{n}{2}\right\rfloor }\left(
-1\right) ^{j}\left( 
\begin{array}{c}
n-j \\ 
j%
\end{array}%
\right) \left( 3x\right) ^{n-2j}2^{j}-\dsum\limits_{j=0}^{\left\lfloor \frac{%
n-1}{2}\right\rfloor }\left( -1\right) ^{j}\left( 
\begin{array}{c}
n-j-1 \\ 
j%
\end{array}%
\right) \left( 3x\right) ^{n-2j}2^{j} \\
&=&\dsum\limits_{j=0}^{\left\lfloor \frac{n}{2}\right\rfloor }\left(
-1\right) ^{j}\frac{n}{n-j}\left( 
\begin{array}{c}
n-j \\ 
j%
\end{array}%
\right) \left( 3x\right) ^{n-2j}2^{j}.
\end{eqnarray*}%
Which is the explicit formula of Mersenne Lucas polynomials.\newline
By using the Eq. (3.3), we easily obtain: 
\begin{eqnarray*}
Gm_{n}\left( x\right) &=&m_{n}\left( x\right) +im_{n-1}\left( x\right) \\
&=&\dsum\limits_{j=0}^{\left\lfloor \frac{n}{2}\right\rfloor }\left(
-1\right) ^{j}\frac{n}{n-j}\left( 
\begin{array}{c}
n-j \\ 
j%
\end{array}%
\right) \left( 3x\right) ^{n-2j}2^{j}+i\tsum\limits_{j=0}^{\left\lfloor 
\frac{n-1}{2}\right\rfloor }\left( -1\right) ^{j}\frac{n-1}{n-j-1}\left( 
\begin{array}{c}
n-j-1 \\ 
j%
\end{array}%
\right) \left( 3x\right) ^{n-2j-1}2^{j}.
\end{eqnarray*}

Which is the explicit formula of Gaussian Mersenne Lucas polynomials.
\end{proof}

\section{\textbf{Conclusion}}

In this study, we defined Gaussian Mersenne Lucas numbers, Mersenne Lucas
polynomials and Gaussian Mersenne Lucas polynomials. Then we gave a formula
for the Gaussian Mersenne Lucas numbers by using the Mersenne Lucas numbers.
The Gaussian Mersenne Lucas polynomials are also given by using the Mersenne
Lucas polynomials. Their Binet's formulas, explicit formulas, generating
functions, symmetric functions and negative extensions are obtained.

\end{document}